\documentclass[twoside, 11pt]{article}
\usepackage{amssymb, amsmath, mathrsfs, amsthm}
\usepackage{graphicx}
\usepackage{color}
\usepackage[top=3cm, bottom=3cm, left=3cm, right=3cm]{geometry}
\usepackage{float, caption, subcaption}
\usepackage{bm}

\input{epsf}

\newcommand{\bd}{\begin{description}}
\newcommand{\ed}{\end{description}}
\newcommand{\bi}{\begin{itemize}}
\newcommand{\ei}{\end{itemize}}
\newcommand{\be}{\begin{enumerate}}
\newcommand{\ee}{\end{enumerate}}
\newcommand{\beq}{\begin{equation}}
\newcommand{\eeq}{\end{equation}}
\newcommand{\beqs}{\begin{eqnarray*}}
\newcommand{\eeqs}{\end{eqnarray*}}

\newcommand{\ceil}[1]{\left\lceil #1 \right\rceil}

\definecolor{DarkGreen}{rgb}{0.2, 0.6, 0.3}

\newtheorem{theorem}{Theorem}[section]

\newtheorem{definition}{Definition}

\begin{document}
\title{\textbf{Gallai-Ramsey numbers for $\ell$-connected graphs} \footnote{Supported by the National
Science Foundation of China (Nos. 12301458, 12401461, 12471329 and 12061059) and the Zhejiang Provincial Natural Science Foundation of China (No. LQ24A010005).  The fourth-listed author is supported by a Scholarly Development Assignment Program grant from Western Carolina University.}}

\author{
Zhao Wang\footnote{College of Science, China Jiliang University, Hangzhou 310018, China. {\tt wangzhao@mail.bnu.edu.cn; zhanglanyanni@163.com}}, \ \ Lanyanni Zhang\footnotemark[2], \ \ Meiqin Wei \footnote{Corresponding Author: School of Science, Shanghai Maritime University, Shanghai 201306, China. {\tt weimeiqin8912@163.com}}, \ \ Mark Budden\footnote{Department of Mathematics and Computer Science, Western Carolina University, Cullowhee, North Carolina, USA. {\tt mrbudden@email.wcu.edu}} }
\date{}
\maketitle

\begin{abstract}
Given a nonempty graph $G$, a collection of nonempty graphs $\mathcal{H}$, and a positive integer $k$, the Gallai-Ramsey number $\mathrm{gr}_k(G:\mathcal{H})$ is defined to be the minimum positive integer $n$ such that every exact $k$-edge-coloring of a complete graph $K_n$ contains either a rainbow copy of $G$ or a monochromatic copy of some element in $\mathcal{H}$.  In this paper, we obtain some exact values and general lower and upper bounds for $\mathrm{gr}_k(G:\mathcal{F}^\ell)$, where $\mathcal{F}^\ell$ is the set of $\ell$-connected graphs and $G\in\{P_5, K_{1,3}\}$.\\[2mm]
\end{abstract}

\section{Introduction}

The graphs considered in this paper are assumed to be nonempty, finite, and undirected simple graphs.  We follow the notations and terminology introduced in the book \cite{BM08}.  If $G$ is a graph, its vertex set is given by $V(G)$ and its edge set is given by $E(G)$.  We denote the order and size of $G$ by $|V(G)|$ and $|E(G)|$, respectively.  For any $n\in \mathbb{N}$, define the notation $[n]:=\{1, 2, \dots , n\}$.
A {\it $k$-edge-coloring} of a graph $G$ is a map $\gamma :E(G)\longrightarrow [k]$.   Such a map is called {\it exact} if it is surjective.  For a given $k$-edge-coloring $\gamma$ of $G$, a subgraph of $G$ is called {\it monochromatic} if all of its edges are in the same color class, and it is called {\it rainbow} if no two of its edges are in the same color class.  If $x\in V(G)$, then we denote by $d_i(x)$ the number of edges in color $i$ that are incident with $x$, where $i\in [k]$.

Let $K_{n}$, $P_{n}$, and $C_{n}$ denote the complete graph, the path, and the cycle of order $n$, respectively.   Let $K_{1,n}$ be a star of size $n$ and let $nK_{2}$ be a matching of size $n$ (i.e., the disjoint union of $n$ copies of $K_2$).  If $G$ is a graph and $\emptyset \ne S\subset V(G)$, then we denote by $G-S$ the subgraph of $G$ induced by $V(G)\setminus S$.   When $S=\{x\}$, we write $G-x$ in place of $G-\{x\}$.  For any graphs $G$ and $H$, the {\it join} of $G$ and $H$, denoted $G\vee H$, is formed by taking the disjoint union of $G$ and $H$, and adding in all of the edges between $V(G)$ and $V(H)$.

Given graphs $H_1, H_2,\dots , H_k$, the {\it Ramsey number} $\mathrm{R}(H_1, H_2,..., H_k)$ is defined to be the least $n\in \mathbb{N}$ such that every $k$-edge-coloring of $K_n$ contains a monochromatic $H_i$ in color $i$, for some $i\in [k]$.  A $k$-edge-coloring of $K_{\mathrm{R}(H_1, H_2,..., H_k)-1}$ that avoids a copy of $H_i$ in color $i$, for all $i\in [k]$, is called a {\it critical coloring} for $\mathrm{R}(H_1, H_2,..., H_k)$.
If $H_1=H_2=\cdot\cdot\cdot=H_k=H$, then we simplify the notation for the Ramsey number to $\mathrm{R}_k(H)$.  More generally, if $\mathcal{H}$ is any nonempty set of graphs, then define the {\it Ramsey number} $\mathrm{R}_k(\mathcal{H})$ to be the minimum integer $n\in \mathbb{N}$ such that every $k$-edge-coloring of $K_n$ contains a monochromatic subgraph that is isomorphic to some graph in $\mathcal{H}$.

A graph $G$ is called {\it $\ell$-connected} if its order is greater than $\ell$ and deleting fewer than $\ell$ vertices results in a connected graph.  Denote by $\mathcal{F}^{\ell}$ the set of all $\ell$-connected graphs.
In 1983, Matula \cite{Ma83} considered $\mathrm{R}_k(\mathcal{F}^{\ell})$, proving the following theorem.
\begin{theorem} [\cite{Ma83}]\label{Thm:Connectivity}
For $k, \ell\geq 2$, we have
$$
2k(\ell-1)+1\leq {\rm R}_k(\mathcal{F}^\ell) < \frac{10}{3}k(\ell-1)+1.
$$
Furthermore,
$$
4(\ell-1)+1\leq {\rm R}_2(\mathcal{F}^\ell) < (3+\sqrt{11/3})(\ell-1)+1.
$$
\end{theorem}
\noindent In this paper, we consider the analogous problem for $k$-edge-colorings that avoid rainbow $P_5$-subgraphs or rainbow $K_{1,3}$-subgraphs.

\begin{definition}
Given a nonempty graph $G$, a collection of nonempty graphs $\mathcal{H}$, and $k\in \mathbb{N}$, the Gallai-Ramsey number $\mathrm{gr}_k(G:\mathcal{H})$ is defined to be the minimum integer $n\in \mathbb{N}$ such that every exact $k$-edge-coloring of $K_n$ contains either a rainbow copy of $G$ or a monochromatic copy of some $H\in \mathcal{H}$.
\end{definition}

This variation of a Ramsey number is named after Tibor Gallai \cite{Gallai} who, in 1967, proved a result equivalent to a classification of rainbow $K_3$-free colorings of complete graphs.  An English translation of \cite{Gallai} by Maffray and Priessmann can be found in \cite{Perfect}, and a statement of Gallai's result in terms of rainbow $K_3$-free colorings can be found in \cite{GyarfasSimonyi} (see also \cite{CameronEdmonds}).  For additional details on Gallai-Ramsey numbers, we refer the reader to \cite{LLS24,LBMWW20,LW20,LS24,LL24,ZW21,ZW22}.

In order for $K_n$ to have an exact $k$-edge-coloring, it is necessary that $${n\choose 2}=\frac{n(n-1)}{2}\ge k,$$ which is equivalent to $n^2-n-2k\ge 0$.  Solving this inequality for $n$, we find that \begin{equation}{\rm gr}_k(G:\mathcal{H})\ge \ceil{\frac{1+\sqrt{1+8k}}{2}},\label{basiclower}\end{equation} for any graph $G$ and any collection of graphs $\mathcal{H}$.
The additional restriction of forbidding a rainbow copy of $G$ renders it evident that ${\rm gr}_k(G:\mathcal{H})\leq {\rm R}_k(\mathcal{H})$.  Also, when $k< |E(G)|$, it is not possible to have a rainbow copy of $G$, and it follows that ${\rm gr}_k(G:\mathcal{H})= {\rm R}_k(\mathcal{H})$.

As we will consider Gallai-Ramsey numbers of the forms ${\rm gr}_k(P_5:\mathcal{F}^\ell)$ and ${\rm gr}_k(K_{1,3}:\mathcal{F}^\ell)$, we must first describe two structural theorems for the corresponding rainbow $P_5$-free and rainbow $K_{1,3}$-free colorings.
Given an edge-colored complete graph, define $V^{j}$ to be the set of vertices that are incident with at least one edge in color $j$, and denote by $E^{j}$ the set of edges with color $j$. Thomason and Wagner \cite{TW07} obtained the following structural theorem for edge-colored complete graphs that avoid a rainbow $P_5$.

\begin{theorem}[\cite{TW07}]\label{Thm:P_5}
Let $n\ge 5$ and consider an edge coloring of $K_n$ that avoids a rainbow $P_5$. Then, up to isomorphism, and a possible reordering of the colors, one of the following holds:
\begin{enumerate}
	\item[$(a)$] At most three colors are used.
	\item[$(b)$] Color $1$ is dominant; that is, the vertex set can be partitioned into $A, V^2, V^3, V^4, \ldots$ such that edges within $V^j$ are colored $1$ or $ j$, edges meeting $A$ and edges between $V^j$ ($j\geq 2$) are colored $1$. This implies that the sets $V^j$ ($j\geq 2$) are disjoint.
	\item[$(c)$] $K_n-a$ is monochromatic for some vertex $a$.
	\item[$(d)$] There are three vertices $a,b,c$ such that $E^{2}=\{ab\}$,
$E^{3}=\{ac\}$, $E^{4}$ contains $bc$ plus perhaps some edges
incident with $a$, and every other edge is in $E^{1}$.
	\item[$(e)$] There are four vertices $a,b,c,d$ such that $\{ab\}\subseteq
E^{2}\subseteq \{ab, cd\}$, $E^{3}=\{ac, bd\}$, $E^{4}=\{ad,
bc\}$ and every other edge is in $E^{1}$.
	\item[$(f)$] $n=5$, $V(K_n) = \{a, b, c, d, e\}$, $E^{1} =\{ad, ae, bc\}$,
$E^{2}=\{bd, be, ac\}$, $E^{3}=\{cd, ce, ab\}$ and
$E^{4}=\{de\}$.
\end{enumerate}
\end{theorem}

Next, we describe the structure of an edge-colored complete graph that avoids a rainbow $K_{1,3}$-subgraph.
For $n\geq 1$, let $G_1(n)$ be a $3$-edge-colored $K_n$ that
satisfies the following conditions: the vertices of $K_n$ are partitioned into three pairwise disjoint sets $V_1$, $V_2$ and $V_3$ such that for $1\leq i\leq 3$ (with indices reduced modulo $3$), all edges
between $V_i$ and $V_{i+1}$ have color $i$, and all edges connecting pairs of vertices within $V_{i+1}$ have color $i$ or $i + 1$. Note that one of $V_1$, $V_2$, and $V_3$ is allowed to be empty, but at
least two of them are nonempty.  Bass et al. \cite{BMOP} considered the edge-colorings of complete graphs without rainbow $K_{1,3}$-subgraphs and obtained the following result.

\begin{theorem}[\cite{BMOP, GLST87}]\label{th-Star-Structure}
For positive integers $k$ and $n$, if $G$ is a $k$-edge-coloring of $K_n$ that avoids a rainbow $K_{1,3}$, then after renumbering the colors, one of the following holds:
\begin{enumerate}
	\item[$(i)$] $k\leq 2$ or $n\leq 3$.
	\item[$(ii)$] $k=3$ and $G\cong G_1(n)$.
	\item[$(iii)$] $k\geq 4$ and Item $(b)$ in Theorem \ref{Thm:P_5} holds.
\end{enumerate}
\end{theorem}

In Section \ref{Sec2}, we first give exact values for ${\rm gr}_{k}(P_5:\mathcal{F}^2)$, when $k\ge 3$. Then we obtain exact values for ${\rm gr}_{k}(P_5:\mathcal{F}^3)$ and ${\rm gr}_{k}(P_5:\mathcal{F}^4)$ when $k\geq 4$, and upper and lower bounds for ${\rm gr}_{k}(P_5:\mathcal{F}^3)$ and ${\rm gr}_{k}(P_5:\mathcal{F}^4)$ when $k=3$.  For $\ell\geq 5$, we give exact values for ${\rm gr}_{k}(P_5:\mathcal{F}^{\ell})$ when $k\geq\ceil{\frac{\ell}{2}}+2$, and upper and lower bounds for ${\rm gr}_{k}(P_5:\mathcal{F}^{\ell})$ when $k\leq\ceil{\frac{\ell}{2}}+1$.  At the end of Section \ref{Sec2}, we obtain the exact values of $\mathrm{gr}_{k}(P_5:\mathcal{F}_n^{\ell})$ when $n\geq \ell+3$, $2\leq \ell\leq 2k-4$, and $n \geq k\geq 5$, where $\mathcal{F}_n^{\ell}$ denotes the set of $\ell$-connected graphs of order $n$.  In Section \ref{Sec3}, we give exact values for ${\rm gr}_{k}(K_{1,3}:\mathcal{F}^{\ell})$ when $k\geq\ceil{\frac{\ell}{2}}+2$, and upper and lower bounds for ${\rm gr}_{k}(K_{1,3}:\mathcal{F}^{\ell})$ when $k\leq\ceil{\frac{\ell}{2}}+1$.

\section{Rainbow \bm{$P_5$}-Free Edge-Colorings}\label{Sec2}

In this section, we consider the Gallai-Ramsey numbers ${\rm gr}_k(P_5:\mathcal{F}^{\ell})$, starting first with the case $\ell =2$.
We will need the following theorem, which is Fact 3.4 in \cite{wei}.

\begin{theorem} [\cite{wei}] \label{Thm:K_5}
Every exact $6$-edge-coloring of $K_5$ contains a rainbow copy of $P_5$.
\end{theorem}

\begin{theorem}\label{Tm:l=2}
For $k\geq 3$, we have
$$
{\rm gr}_{k}(P_5:\mathcal{F}^2) = \begin{cases}
7 & \text{if $k=3$},\\
6 & \text{if $k=4$},\\
5 & \text{if $k=5, 6$},\\
\ceil{\frac{1+\sqrt{1+8k}}{2}} & \text{if $k\geq 7$}.\\
\end{cases}
$$
\end{theorem}

\begin{proof} We break the proof into cases, based on the value of $k$.

\underline{Case 1:} Assume that $k=3$.
Note that ${\rm gr}_3(P_5:\mathcal{F}^2)={\rm R}_3(\mathcal{F}^2)$ since four colors are necessary to have a rainbow $P_5$.  The lower bound $\mathrm{R}_3(\mathcal{F}^{2})\ge 7$ is given by Theorem \ref{Thm:Connectivity}.
To prove that $7$ is also an upper bound, consider an exact $3$-coloring of $K_n$, where $n\ge 7$.  Since $|E(K_n)|=\frac{n(n-1)}{2}$ and the edges are partitioned among three colors, by the Pigeonhole Principle, some color class contains at least $\ceil{\frac{n(n-1)}{6}}$ edges.  Since $n\ge 7$, it follows that $\ceil{\frac{n(n-1)}{6}}\ge n$.  Without loss of generality, suppose that there are at least $n$ edges in color $1$.  So, the subgraph of $K_n$ spanned by the edges in color $1$ cannot be a tree or a forest, as such a graph would have size at most $n-1$.  Hence, there exists a cycle in color $1$, which is a $2$-connected graph.  It follows that $\mathrm{R}_3(\mathcal{F}^{2})\le 7$.

\underline{Case 2:} Assume that $k=4$.
Let $\gamma$ be the exact $4$-edge-coloring of $K_5$ with vertex set $V(K_5)=\{v_i:1\leq i\leq 5\}$ such that $\gamma(v_1v_2)=\gamma(v_3v_4)=2$, $\gamma(v_1v_3)=\gamma(v_2v_4)=3$, $\gamma(v_1v_4)=\gamma(v_2v_3)=4$, and $\gamma(v_1v_5)=\gamma(v_2v_5)= \gamma(v_3v_5)=\gamma(v_4v_5)=1$.
It is easy to check there is neither a rainbow copy of $P_5$ nor a monochromatic $2$-connected subgraph, which implies that $\mathrm{gr}_{4}(P_5:F^2)\geq 6$.  To show that $\mathrm{gr}_{4}(P_5:F^2)\leq 6$, consider an exact $4$-edge-coloring of $K_n$ ($n\ge 6$) that avoids a rainbow $P_5$.
By Theorem \ref{Thm:P_5}, one of the conditions $(b)$, $(c)$, $(d)$, or $(e)$ must hold.  If $(b)$ holds, then $|V^2|\geq 2$ and $|V^3|\geq 2$.  Since all edges between $V^2$ and $V^3$ are colored $1$, there is a $C_4$ in color $1$, which is a $2$-connected subgraph.
If $(c)$ holds, then $K_6-a$ is monochromatic for some vertex $a$.  It follows that
there is a monochromatic $K_5$, which necessarily contains a monochromatic $2$-connected subgraph.
If $(d)$ or $(e)$ holds, we can find a monochromatic $K_2\vee \overline{K}_2$ or $K_2\vee \overline{K}_4$ in color $1$, respectively.  Thus, in all cases, there is a monochromatic $2$-connected subgraph, so $\mathrm{gr}_{4}(P_5:F^2)\leq 6$.

\underline{Case 3:} Assume that $k\in \{5,6\}$.
Every exact $5$ or $6$-edge-coloring of $K_4$ has at most two colors in any given color class, so a monochromatic $2$-connected subgraph is avoided.  There also are not enough vertices to have a rainbow $P_5$.  Thus, $\mathrm{gr}_{k}(P_5:F^2)\geq 5$ when $k\in \{5,6\}$.   To prove that $5$ is also an upper bound, we first handle the $k=5$ case.  Consider an exact $5$-edge-coloring of $K_n$ ($n\ge 5$) that avoids a rainbow $P_5$.  From Theorem \ref{Thm:P_5}, one of $(b)$ or $(c)$ must hold.   If $(b)$ holds, then $\min\{|V^2|, |V^3|,|V^4|,|V^5|\}\geq 2$. Since all edges between $V^i$ and $V^j$ $(2\leq i\neq j\leq 5)$ are colored $1$, there is a monochromatic $C_4$ in color $1$.  If (c) holds, then there exists a monochromatic $K_4$, which contains a $C_4$.  Thus, ${\rm gr}_5(P_5:\mathcal{F}^2)\le 5$.
If $k=6$, then consider an exact $6$-edge-coloring of $K_n$ ($n\ge 5$).   In $n=5$, then Theorem \ref{Thm:K_5} implies that every $6$-edge-coloring of $K_5$ contains a rainbow $P_5$.  If $n\ge 6$, then from Theorem \ref{Thm:P_5}, either $(b)$ or $(c)$ holds.  Similar to the $k=5$ case, we can find a monochromatic $2$-connected subgraph.
It follows that ${\rm gr}_6(P_5:\mathcal{F}^2)\le 5$.

\underline{Case 4:} Assume that $k\ge 7$.
For $k\geq 7$, we have $\mathrm{gr}_{k}(P_5:\mathcal{F}^2) \geq \ceil{\frac{1+\sqrt{1+8k}}{2}},$ by Inequality (\ref{basiclower}).
To prove that $\mathrm{gr}_{k}(P_5:\mathcal{F}^2) \leq \ceil{\frac{1+\sqrt{1+8k}}{2}}$, let $\gamma$ be an exact $k$-edge-coloring of $K_n$, where $n\geq \ceil{\frac{1+\sqrt{1+8k}}{2}}$, that avoids a rainbow $P_5$.  From Theorem \ref{Thm:P_5}, either $(b)$ or $(c)$ holds. Similar to the previous case, we can find monochromatic $2$-connected subgraphs under this edge coloring $\gamma$.
\end{proof}

Now we consider the collection of $3$-connected graphs.   When $k=3$, Theorem \ref{Thm:Connectivity} implies that $\mathrm{gr}_{3}(P_5:\mathcal{F}^3)=\mathrm{R}_3(\mathcal{F}^3) \in [13, 20]$.  In the next theorem, we consider the cases where $k\ge 4$.

\begin{theorem}\label{th-F3}
For $k\geq 4$, we have $$\mathrm{gr}_{k}(P_5:\mathcal{F}^3) = \begin{cases}
7 & \text{if $k=4$},\\
5 & \text{if $k=5, 6$},\\
\ceil{\frac{1+\sqrt{1+8k}}{2}} & \text{if $k\geq 7$}.\\
\end{cases}$$
\end{theorem}

\begin{proof} Consider the following cases.

\underline{Case 1:} Assume that $k=4$.
Let $\gamma$ be the $4$-edge-coloring of $K_6$ with vertex set $V(K_6)=\{v_1,v_2,\dots , v_6\}$ such that $\gamma(v_1v_2)=\gamma(v_3v_4)=2$, $\gamma(v_1v_3)=\gamma(v_2v_4)=3$, $\gamma(v_1v_4)=\gamma(v_2v_3)=4$, and $\gamma(v_1v_5)=\gamma(v_1v_6)=\gamma(v_2v_5)=
\gamma(v_2v_6)=\gamma(v_3v_5)=\gamma(v_3v_6)=
\gamma(v_4v_5)=\gamma(v_4v_6)=\gamma(v_5v_6)=1$.
It is easy to check that under the coloring $\gamma$, there is neither a rainbow copy of $P_5$ nor a monochromatic $3$-connected subgraph, which means that $\mathrm{gr}_{4}(P_5:\mathcal{F}^3)\geq 7$.  To prove that $\mathrm{gr}_{4}(P_5:\mathcal{F}^3)\leq 7$, consider an exact $4$-edge-coloring of $K_n$ ($n\geq 7$) that avoids a rainbow copy of $P_5$.  It follows from Theorem \ref{Thm:P_5} that $(b)$, $(c)$, $(d)$, or $(e)$ holds. If $(b)$ holds, then $\min\{|V^2|,|V^3|,|V^4|\}\geq 2$. Since all edges between $V^i$ to $V^j$ ($2\leq i\neq j\leq 4$) are colored $1$, and hence there is a monochromatic $3$-connected subgraph with color $1$.  If $(c)$ holds, then $K_n$ contains a monochromatic $K_5$, which means $K_n$ has a monochromatic $3$-connected subgraph.  If $(d)$ or $(e)$ hold, we can find a $K_3\vee\overline{K}_2$ or a $K_3\vee\overline{K}_4$ in color $1$, respectively.  It follows that $\mathrm{gr}_{4}(P_5:\mathcal{F}^3)\leq 7$.

\underline{Case 2:} Assume that $k\in \{5, 6\}$.  Every exact $5$ or $6$-edge-coloring of $K_4$ has at most two colors in any given color class, so a monochromatic $3$-connected subgraph is avoided.  There also are not enough vertices to have a rainbow $P_5$.  Thus, $\mathrm{gr}_{k}(P_5:\mathcal{F}^3)\geq 5$ when $k\in \{5,6\}$.
To prove that $\mathrm{gr}_{k}(P_5:\mathcal{F}^3)\leq 5$ when $k\in \{5,6\}$, consider an exact $6$-edge-coloring of $K_n$ ($n\geq 5$) that avoids a rainbow copy of $P_5$.  From Theorem \ref{Thm:P_5}, either $(b)$ or $(c)$ holds.  For $n\geq 6$, if $(b)$ holds, then $\min\{|V^2|,|V^3|,|V^4|,|V^5|\}\geq 2$. Since all edges between $V^i$ and $V^j$ ($2\leq i\neq j\leq 5$) are colored $1$, there is a monochromatic $3$-connected subgraph with color $1$.  If $(c)$ holds, then $K_n$ contains a monochromatic $K_4$, which means that $K_n$ has a monochromatic $3$-connected subgraph. For $n=5$, it follows from Theorem \ref{Thm:K_5} that any $6$-edge colored $K_5$ contains a rainbow copy of $P_5$, a contradiction. Hence, we have $\mathrm{gr}_{6}(P_5:\mathcal{F}^3)\leq 5$.

\underline{Case 3:} Assume that $k\geq 7$.  From Inequality (\ref{basiclower}), it follows that $\mathrm{gr}_{k}(P_5:\mathcal{F}^3) \geq \ceil{\frac{1+\sqrt{1+8k}}{2}}$.
Now we prove $\mathrm{gr}_{k}(P_5:\mathcal{F}^3) \leq \ceil{\frac{1+\sqrt{1+8k}}{2}}$.
Let $G$ be an exact $k$-edge-colored $K_n$, where $n\geq \ceil{\frac{1+\sqrt{1+8c}}{2}}$, that avoids a rainbow $P_5$.  From Theorem \ref{Thm:P_5}, we have either $(b)$ or $(c)$ holds.  With similar analysis as in the proof of the previous theorem, we can conclude that $G$ always contains a monochromatic $3$-connected subgraph.
\end{proof}

Next, we consider $4$-connected graphs.  When $k=3$, Theorem \ref{Thm:Connectivity} implies that $\mathrm{gr}_{3}(P_5:\mathcal{F}^4)=R_3(F^4) \in [19, 30]$.  In the next theorem, we consider the value of $\mathrm{gr}_{k}(P_5:\mathcal{F}^4)$ when $k\ge 4$.

\begin{theorem}
For $k\geq 4$, we have
$$
\mathrm{gr}_{k}(P_5:\mathcal{F}^4) = \begin{cases}
8 & \text{if $k=4$},\\
6 & \text{if $k=5$},\\
5 & \text{if $k=6$},\\
\ceil{\frac{1+\sqrt{1+8k}}{2}} & \text{if $k\geq 7$}.\\
\end{cases}
$$
\end{theorem}
\begin{proof}
Consider the following cases.

\underline{Case 1:} Assume that $k=4$.
Let $\gamma_1$ be the $4$-edge-coloring of $K_7$ with vertex set $V(K_7)=\{v_i:1\leq i\leq 7\}$ such that $\gamma_1(v_1v_2)=\gamma_1(v_3v_4)=2$, $\gamma_1(v_1v_3)=\gamma_1(v_2v_4)=3$, $\gamma_1(v_1v_4)=\gamma_1(v_2v_3)=4$ and all other edges are colored 1. It is easy to check that there is neither a rainbow copy of $P_5$ nor a monochromatic $4$-connected subgraph under the coloring $\gamma$, which means that $\mathrm{gr}_{4}(P_5:\mathcal{F}^4)\geq 8$. To prove that $\mathrm{gr}_{4}(P_5:\mathcal{F}^4)\leq 8$, consider an exact $4$-edge-coloring of $K_n$ ($n\geq 8$) that avoids arainbow copy of $P_5$. It follows from Theorem \ref{Thm:P_5} that $(b)$, $(c)$, $(d)$, or $(e)$ holds. If $(b)$ holds, then $\min\{|V^2|,|V^3|,|V^4|\}\geq 2$. Since all edges between $V^i$ and $V^j$ ($2\leq i\neq j\leq 4$) are colored $1$, and hence there is a monochromatic $4$-connected subgraph with color $1$. If $(c)$ holds, then $K_n$ contains a monochromatic $K_7$, which means that $K_n$ contains a monochromatic $4$-connected subgraph. If $(d)$ or $(e)$ holds, we can find $K_4\vee\overline{K}_2$ or $K_4\vee\overline{K}_4$ of color $1$ respectively. Therefore, $K_n$ contains a monochromatic $4$-connected subgraph of color $1$ and then $gr_{4}(P_5:\mathcal{F}^4)\leq 8$.

\underline{Case 2:} Assume that $k=5$.
Let $\gamma_2$ be the $5$-edge-coloring of $K_5$ with vertex set $V(K_5)=\{v_i:1\leq i\leq 5\}$ such that $\gamma_2(v_1v_2)=1$, $\gamma_2(v_1v_3)=2$, $\gamma_2(v_1v_4)=3$, $\gamma_2(v_1v_5)=4$ and all other edges are colored $5$. It is easy to check that there is neither a rainbow copy of $P_5$ nor a monochromatic $4$-connected subgraph under the coloring $\gamma$, which means that $gr_{5}(P_5:\mathcal{F}^4)\geq 6$. To prove that$\mathrm{gr}_{5}(P_5:\mathcal{F}^4)\leq 6$, consider an exact $5$-edge-coloring of $K_n$ ($n\geq 6$) that avoids a rainbow copy of $P_5$.  It follows from Theorem \ref{Thm:P_5} that $(b)$ or $(c)$ holds. If $(c)$ holds, then $G$ contains a monochromatic $K_5$, which means that $G$ contains a monochromatic $4$-connected subgraph. If $(b)$ holds, then $\min\{|V^2|,|V^3|,|V^4|,|V^5|\}\geq 2$. Since all edges between $V^i$ and $V^j$ ($2\leq i\neq j\leq 5$) are colored $1$, there is a monochromatic $4$-connected subgraph with color $1$. Hence, $\mathrm{gr}_{5}(P_5:\mathcal{F}^4)\leq 6$.

\underline{Case 3:} Assume that $k=6$.  In an exact $6$-edge-coloring of $K_4$, every edge receives its own color.   So, no monochromatic $4$-connected subgraph exists and there are not enough vertices to have a rainbow $P_5$.  It follows that $\mathrm{gr}_6(P_5 : \mathcal{F}^4)\ge 5$.  To prove that $\mathrm{gr}_{k}(P_5:\mathcal{F}^4)\leq 5$, let $H$ be an exact $6$-edge-colored $K_n$ ($n\geq 5$) that lacks a rainbow copy of $P_5$.  From Theorem \ref{Thm:P_5}, either $(b)$ or $(c)$ holds.  For $n\geq 6$, if $(b)$ holds, then $\min\{|V^2|,|V^3|,|V^4|,|V^5|,|V^6|\}\geq 2$. Since all edges between $V^i$ and $V^j$ ($2\leq i\neq j\leq 6$) are colored $1$, there is a monochromatic $4$-connected subgraph with color $1$.  If $(c)$ holds, then $H$ contains a monochromatic $K_5$, which means that $H$ has a monochromatic $4$-connected subgraph. For $n=5$, it follows from Theorem \ref{Thm:K_5} that any exact $6$-edge -coloring of $K_5$ contains a rainbow copy of $P_5$, giving a contradiction.  Hence, we have $\mathrm{gr}_{6}(P_5:\mathcal{F}^3)\leq 5$.

\underline{Case 4:} Assume that $k\geq 7$.
By Inequality (\ref{basiclower}), we have that $\mathrm{gr}_{k}(P_5:\mathcal{F}^4) \geq \ceil{\frac{1+\sqrt{1+8k}}{2}}$.
To prove that $\mathrm{gr}_{k}(P_5:\mathcal{F}^4) \leq \ceil{\frac{1+\sqrt{1+8k}}{2}}$, let $G$ be an exact $k$-edge-colored $K_n$,  where $n\geq \ceil{\frac{1+\sqrt{1+8k}}{2}}$, without a rainbow copy of $P_5$.  We know from Theorem \ref{Thm:P_5} that $(b)$ or $(c)$ holds. With a similar analysis to that of the proof of the previous theorem, we come to the conclusion that $G$ contains a monochromatic $4$-connected subgraph.
\end{proof}

Now we consider the values of $\mathrm{gr}_k(P_5:\mathcal{F}^\ell)$when $\ell\geq 5$.  When $k=3$, ${\rm gr}_3(P_5:\mathcal{F}^\ell)=R_3(\mathcal{F}^\ell) \in [6\ell-5,10\ell-10]$ by Theorem \ref{Thm:Connectivity}.  The following theorem considers the values of ${\rm gr}_k(P_5:\mathcal{F}^\ell)$ when $k\ge 4$ and $\ell \ge 4$.

\begin{theorem}\label{Tm:k=l}  Suppose that $\ell \ge 5$.
If $4\le k\le \ceil{\frac{\ell}{2}}+1$, then $$2\ell\le {\rm gr}_k(P_5:\mathcal{F}^\ell)\le (3+\sqrt{11/3})(\ell-1).$$  If $k\ge \ceil{\frac{\ell}{2}}+2$, then
$$
{\rm gr_{k}}(P_5:F^\ell) = \begin{cases}
 \ell+2 & \text{if $\ceil{\frac{\ell}{2}} +2\leq k\leq \ell+1$},\\
\ceil{\frac{1+\sqrt{1+8k}}{2}} & \text{if $k\geq\ell+2$}.
\end{cases}
$$
\end{theorem}

\begin{proof} Consider the following cases.

\underline{Case 1:} Assume that $4\le k\le \ceil{\frac{\ell}{2}}+1$.
For the lower bound, partition the vertex set $V(K_{2\ell-1})$ into $k-1$ sets $V_2, \cdots, V_k$ such that $|V_2|=\ell$, $\min\{|V_3|,|V_4|,\cdots,|V_k|\}\geq 2$, and $|V_3|+|V_4|+\cdots+|V_k|=\ell-1$.  Color the edges within $V_i$ ($2\leq i\leq k$) with color $i$, and for each pair $V_i$, $V_j$ ($2\leq i\neq j\leq k$), color all of the edges between $V_i$ and $V_j$ with color $1$. Since the resulting exact $k$-edge-coloring of $K_{2\ell-1}$ avoids a rainbow copy of $P_5$ and a monochromatic $\ell$-connected subgraph, it follows that  $\mathrm{gr}_{k}(P_5:\mathcal{F}^\ell)\geq 2\ell$.

To prove the upper bound, consider an exact $k$-edge-coloring of $K_n$ ($n\geq (3+\sqrt{11/3})(\ell-1)\geq 6$) that avoids a rainbow $P_5$ and denote this coloring by $G$.
We know from Theorem \ref{Thm:P_5} that $(b)$, $(c)$, $(d)$, or $(e)$ holds. If $(b)$ holds, there is a partition of $V(G)$, say $V_1, V_2, \cdots,V_k$, where the edges among $V_i$ ($1\leq i\leq k$) are colored $1$ or $i$. For each pair $V_i$, $V_j$ ($1\leq i\neq j\leq k$), all edges between $V_i$ and $V_j$ receive the color $1$. Obviously, $|V_i|\geq 2$ ($2\leq i\leq k$). If we re-color all edges of color $i$ within parts $V_i$ ($3\leq i\leq k$) to color $2$, then all edges in $G$ will be colored only with color $1$ or $2$. Note that if there is a monochromatic $\ell$-connected subgraph after re-coloring, it must also exist before re-coloring. From Theorem \ref{Thm:Connectivity}, we have ${\rm R}_k(\mathcal{F}^\ell)\le (3+\sqrt{11/3})(\ell-1)$.  It follows that $G$ contains  a monochromatic $\ell$-connected subgraph after re-coloring.  If $(c)$, $(d)$, or $(e)$ holds, then $G$ contains a monochromatic $K_{(3+\sqrt{11/3})(\ell-1)-3}$, which means that $G$ contains a monochromatic $\ell$-connected subgraph. Hence, $\mathrm{gr}_{k}(P_5:\mathcal{F}^\ell)\leq (3+\sqrt{11/3})(\ell-1)$.

\underline{Case 2:} Assume that $\ceil{\frac{\ell}{2}} +2\leq k\leq \ell+1$.
For the lower bound, begin with a monochromatic $K_{\ell}$ in color $1$, introduce a vertex $x$, and arbitrarily color the edges joining $x$ to the $K_\ell$ using all of the colors $2,3,\dots , k$. This is possible since $k-1\le \ell$ in this case.  The resulting exact $k$-edge-coloring of $K_{\ell+1}$ avoids a rainbow $P_5$ and a monochromatic $\ell$-connected subgraph, from which it follows that $\mathrm{gr}_{k}(P_5:\mathcal{F}^\ell) \geq \ell+2$.

To prove that $\mathrm{gr}_{k}(P_5:\mathcal{F}^\ell) \leq \ell+2$, consider an exact $k$-edge-coloring of $K_n$ ($n\geq \ell+2$) that avoids a rainbow $P_5$, and denote this coloring by $G$.
Since $k\geq\ceil{\frac{\ell}{2}}+2\geq 5$, by Theorem \ref{Thm:P_5}, either $(b)$ or $(c)$ holds. If $(c)$ holds, then $G$ contains a monochromatic $K_{\ell+1}$, which implies there is a monochromatic $\ell$-connected subgraph. If $(b)$ holds, without loss of generality, we suppose that $|V_2|\geq |V_3|\geq \cdots\geq |V_k|\geq 2$. Since $k\geq\ceil{\frac{\ell}{2}}+2$, we have $2(k-2)\geq \ell$ and then $|V_2|+|V_3|+\cdots +|V_{k}|\geq 2(k-1)\geq\ell+2$. For each $i\in \{2, 3, \ldots, k\}$, select two vertices $v_{1i},v_{2i}\in V_i$ and set $S=\bigcup_{i=2}^k\{v_{1i},v_{2i}\}$. It follows that $S\geq \ell+2$, and then the set $S$ induces a monochromatic $\ell$-connected subgraph in color $1$ in $G$.  It follows that $\mathrm{gr}_{k}(P_5:\mathcal{F}^\ell)\leq \ell+2$.

\underline{Case 3:} Assume that $k\geq\ell+2$.
From Inequality (\ref{basiclower}), it follows that $\mathrm{gr}_{k}(P_5:\mathcal{F}^\ell)\geq\ceil{\frac{1+\sqrt{1+8k}}{2}}$.
To prove that $\mathrm{gr}_{k}(P_5:\mathcal{F}^\ell)\leq\ceil{\frac{1+\sqrt{1+8k}}{2}}$, let $G$ be an exact $k$-edge-colored $K_n$, where $n\geq \ceil{\frac{1+\sqrt{1+8k}}{2}}$, that avoids a rainbow $P_5$. By Theorem \ref{Thm:P_5}, either $(b)$ or $(c)$ holds.  If $(c)$ holds, then $G$ contains a monochromatic $K_{n-1}$. Since $k\geq \ell+2$, we have $n-1\geq k-1 \geq \ell+1$, which means $G$ contains a monochromatic $\ell$-connected subgraph.  If $(b)$ holds, with loss of generality, we suppose that $|V_2|\geq |V_3|\geq\cdots\geq|V_k|\geq 2$. Since $k\geq \ell+2$, we have $|V_2|+|V_3|+\cdots +|V_{k}|\geq 2(k-2)\geq 2\ell$. For each $i\in \{2, 3, \ldots, k\}$, select two vertices $v_{1i},v_{2i}\in V_i$ and set $S=\bigcup_{i=2}^k\{v_{1i},v_{2i}\}$. It follows that $S\geq 2\ell$ and then the set $S$ induces a monochromatic $\ell$-connected subgraph with color $1$ in $G$. Hence, $\mathrm{gr}_{k}(P_5:\mathcal{F}^\ell)\leq\ceil{\frac{1+\sqrt{1+8k}}{2}}$.
\end{proof}

In the next theorem, we consider the Gallai-Ramsey numbers $\mathrm{gr}_{k}(P_5:\mathcal{F}_n^{\ell})$, where $\mathcal{F}_n^{\ell}$ is the subset of $\mathcal{F}^\ell$ consisting of all $\ell$-connected subgraphs of order $n$.  Note that $${\rm gr}_k(P_5:\mathcal{F}^\ell_n)\ge {\rm gr}_k(P_5:\mathcal{F}^\ell),$$ for all $n\ge \ell+1\ge 3$.

\begin{theorem}
For $n\geq \ell+3$, $2\leq \ell\leq 2k-4$, and $n \geq k\geq 5$, $$ \mathrm{gr}_{k}(P_5:\mathcal{F}_n^{\ell}) =n+\ell-1.$$
\end{theorem}
\begin{proof}
Let $V(K_{n+\ell-2})=\{u_i\mid 1\leq i\leq n-1\}\bigcup \{v_j \mid 1\leq j\leq \ell-1\}$ and let $G$ be a $k$-edge-colored $K_{n+\ell-2}$ such that $E^1=\{u_iu_l \mid 1\leq i\neq l\leq n-1\}\bigcup \{v_jv_t \mid 1\leq j\neq t\leq \ell-1\}$, $E^2=\{u_1v_j \mid 1\leq j\leq \ell-1\}$, $E^3=\{u_2v_j \mid 1\leq j\leq \ell-1\}$, $\cdot\cdot\cdot$, $E^{k-1}=\{u_{k-2}v_j \mid 1\leq j\leq \ell-1\}$ and $E^k=\bigcup_{j=1}^{\ell-1}\bigcup_{i=k-1}^{n-1}\{u_{i}v_j\}$. It is easy to verify that $G$ lacks a rainbow $P_5$ and a monochromatic $\ell$-connected subgraph of order $n$ since $n\geq \ell+3$.  Hence, $\mathrm{gr}_{k}(P_5:\mathcal{F}_n^{\ell})\geq n+\ell-1$.

Now we show that $\mathrm{gr}_{k}(P_5:\mathcal{F}_n^{\ell})\leq n+\ell-1$.
Suppose that $G$ be an exact $k$-edge-colored $K_m$ $(m\geq n+\ell-1)$ that avoids a rainbow $P_5$. According to Theorem \ref{Thm:P_5}, either $(b)$ or $(c)$ holds. If $(c)$ holds, then $G$ contains a monochromatic $K_{n+\ell-2}$, which means that $G$ contains a monochromatic member of $\mathcal{F}_n^{\ell}$. If $(b)$ holds, without loss of generality, we suppose that $|V_2|\geq |V_3|\geq \cdots\geq |V_k|\geq 2$.
Note that $|V_3|+\cdots +|V_k|\geq 2(k-2) \geq \ell$. For each $i\in \{3, \ldots, k\}$, select two vertices $\{v_{1i},v_{2i}\}\subseteq \{V_i\}$. Then, arbitrarily choose $n-(|V_3|+\cdots +|V_k|)$ vertices from $V_i$ ($2\leq i\leq k$), ensuring all $n$ vertices are distinct. These $n$ vertices form an $\ell$-connected subgraph in color $1$. So $G$ contains a monochromatic copy of $\mathcal{F}_n^{\ell}$ with color $1$. Hence, $\mathrm{gr}_{k}(P_5:\mathcal{F}_n ^\ell)\leq n+\ell-1$.
\end{proof}

\section{Rainbow \bm{$K_{1,3}$}-Free Edge-Colorings}\label{Sec3}

In this section, we prove the following theorem concerning Gallai-Ramsey numbers of the form ${\rm gr}_k(K_{1,3}:\mathcal{F}^\ell)$.

\begin{theorem}\label{K13}
Suppose that $\ell \geq 3$.  If $k=3$, then $$4(\ell-2)+3\le {\rm gr}_3(K_{1,3}:\mathcal{F}^\ell)\le (4+\sqrt{11/3})(\ell-1).$$  If $4\leq k\leq \lceil \frac{\ell}{2} \rceil +1$, then $$2\ell\le {\rm gr}_k(K_{1,3}:\mathcal{F}^\ell)\le (3+\sqrt{11/3})(\ell-1).$$  If $k\geq\lceil\frac{\ell}{2} \rceil +2$, then $${\rm gr}_k(K_{1,3}:\mathcal{F}^\ell)=\ceil{\frac{1+\sqrt{1+8k}}{2}}.$$
\end{theorem}

\begin{proof} Consider the following cases.

\underline{Case 1:} Assume that $k=3$.
Let $G$ be the exact $3$-edge-colored $K_{4(\ell-2)+2}$ with vertex set $\{v_i: 1\leq i\leq 4(\ell-2)+2\}$ and with edges colored as follows. Color $v_1v_2$ with color $3$, all other edges incident to $v_1$ (resp. $v_2$) with color $1$ (respectively, color $2$), and all other edges with colors $1$ or $2$.  According to Theorem \ref{Thm:Connectivity}, $G-\{v_1,v_2\}$ contains neither an $(\ell-1)$-connected subgraph with color $1$ nor an $(\ell-1)$-connected subgraph with color $2$. Therefore, $G$ contains neither a rainbow $K_{1,3}$ nor a monochromatic $\ell$-connected subgraph, which means that $\mathrm{gr}_{k}(K_{1,3}:\mathcal{F}^\ell)\geq 4(\ell-2)+3$.

Suppose that $H$ is an exact $3$-edge-colored $K_n$ ($n\geq (4+\sqrt{11/3})(\ell-1)$) that lacks a rainbow copy of $K_{1,3}$. We know from Theorem \ref{th-Star-Structure} that $(ii)$ is true. Without loss of generality, we may assume that $|V_1|\geq |V_2|\geq |V_3|$. Suppose that $|V_3|\leq |V_2|\leq \ell-1$. Otherwise, $|V_1|\geq |V_2|\geq\ell$ and then the induced subgraph $H[V_1\cup V_2]$ contains a monochromatic $K_{\ell,\ell}$ with color $1$, which is a monochromatic $\ell$-connected subgraph.  It follows from $|V_3|\leq \ell-1$ that $|V_1|+|V_2|\geq (3+\sqrt{11/3})(\ell-1)$. Replace the color $2$ edges within $H[V_2]$ to color $3$ and then after the replacement, $H\setminus V_3$ contains a monochromatic $\ell$-connected subgraph by Theorem \ref{Thm:Connectivity}. As a result, $H\setminus V_3$ contains a monochromatic $\ell$-connected subgraph before the replacement. Therefore, $H$ always contains a monochromatic $\ell$-connected subgraph and thus $\mathrm{gr}_{k}(K_{1,3}:\mathcal{F}^\ell)\leq (4+\sqrt{11/3})(\ell-1)$.

\underline{Case 2:} Assume that $4\leq k\leq \lceil \frac{\ell}{2} \rceil +1$.
To prove the lower bound, partition the vertex set $V(K_{2\ell-1})$ into $k-1$ sets $V_2, \cdots, V_k$ such that $|V_2|=\ell$, $\min\{|V_3|,|V_4|,\cdots,|V_k|\}\geq 2$, and $|V_3|+|V_4|+\cdots+|V_k|=\ell-1$. Color the edges within $V_i$ ($2\leq i\leq k$) with color $i$, and for each pair $V_i$, $V_j$ ($2\leq i\neq j\leq k$), color all of the edges between $V_i$ and $V_j$ with color $1$. Since the resulting exact $k$-edge-coloring of $K_{2\ell-1}$ avoids a rainbow copy of $K_{1,3}$ and a monochromatic $\ell$-connected subgraph, it follows that  $\mathrm{gr}_{k}(K_{1,3}:\mathcal{F}^\ell)\geq 2\ell$.

To prove the upper bound, let $G$ be any $k$-edge-colored $K_n$ ($n\geq (3+\sqrt{11/3})(\ell-1)> 6$) that avoids a rainbow $K_{1,3}$. By Theorem \ref{th-Star-Structure}(iii), there is a partition of $V(G)$, say $V_1, V_2, \cdots, V_k$, where $V_1$ may be empty and $\min\{|V_2|,\dots,|V_k|\}\geq 2$. The edges among $V_i$ ($1\leq i\leq k$) are colored $1$ or $i$, and for each pair $V_i$, $V_j$ ($1\leq i\neq j\leq k$), all the edges between $V_i$ and $V_j$ receive the color $1$. If we re-color all edges of color $i$ within parts $V_i$ ($3\leq i\leq k$) to color $2$, then the resulting graph, denoted by $G'$, contains a monochromatic $\ell$-connected subgraph since ${\rm R}_2(F^\ell)<(3+\sqrt{11/3})(\ell-1)+1$. It follows that $G$ also contains a monochromatic $\ell$-connected subgraph and thus $\mathrm{gr}_{k}(K_{1,3}:\mathcal{F}^\ell)\leq(3+\sqrt{11/3})(\ell-1)$.

\underline{Case 3:} Assume that $k\geq\lceil\frac{\ell}{2} \rceil +2$.
From Inequality (\ref{basiclower}), we have that $\mathrm{gr}_{k}(K_{1,3}:\mathcal{F}^\ell) \geq \ceil{\frac{1+\sqrt{1+8k}}{2}}$.
To prove that $\mathrm{gr}_{k}(K_{1,3}:\mathcal{F}^\ell) \leq \ceil{\frac{1+\sqrt{1+8k}}{2}}$, let $G$ be an exact $k$-edge-colored $K_n$, where $n\geq \ceil{\frac{1+\sqrt{1+8k}}{2}}$, that avoids a rainbow $K_{1,3}$.
By Theorem \ref{th-Star-Structure} (iii), there is a partition of $V(G)$, say $V_1, V_2, \cdots, V_k$, where $V_1$ may be empty and $\min\{|V_2|,\dots,|V_k|\}\geq 2$. The edges among $V_i$ ($1\leq i\leq k$) are colored $1$ or $i$, and for each pair $V_i$, $V_j$ ($1\leq i\neq j\leq k$), all the edges between $V_i$ and $V_j$ receive the color $1$.
Therefore, we have $\mathrm{gr}_{k}(K_{1,3}:\mathcal{F}^\ell) =\ceil{\frac{1+\sqrt{1+8k}}{2}}$.
\end{proof}

\end{document}